\documentclass{amsart}
\usepackage{amssymb}
\usepackage{tikz}
\usepackage{epic,curves,mfpic}
\newtheorem{theo}{Theorem}[section]

\newtheorem{conj}[theo]{Conjecture}

\newtheorem{prop}[theo]{Proposition}

\theoremstyle{remark}

\newtheorem{defi}[theo]{Definition}

\newcommand{\Z}{\mathbb{Z}}

\newcommand{\cz}{\mathcal{Z}}

\newcounter{fig}
\begin{document}

\title{The Kropholler Conjecture}
\author{M.J.Dunwoody}

\begin {abstract}
The first version of this paper, gave another
 proof of the Kropholler Conjecture, which gives a relative version of Stallings' Ends Theorem, following an earlier incorrect proof. 
It has been pointed out by Sam Shepherd that the the second proof was still inadequate.   We explain the difficulty and possible ways to obtain a correct proof.
 
\end {abstract}

\maketitle

\section{Introduction} In \cite{D}  the I gave the first of two incorrect proofs of the long standing Kropholler  Conjecture.
   Unfortunately the proof was inadequate, as was pointed out by Alex Margolis in a careful analysis.
The mistake was on p168 in the discussion about fitting two trees together.  In the first version of this paper I attempted another proof, but again it was inadequate.

\section {The Kropholler Conjecture}
  We follow the notation and terminology of \cite {D}, which contains a historical account of the conjecture and the results already obtained.
  
  Let $G$  be a group  and let $H$ be a subgroup.   A subset $S \subset G$ is said to be $H$-finite if $S$ is contained in finitely many
  right cosets of $H$, i.e. for some finite subset $F$ of $G$, we have $S \subset HF$.
  A subgroup $K$ of $G$ is $H$-finite,  if and only if $H\cap K$ has finite index in $K$. 
  Two subsets $R, S \subset G$ are said to be $H$-almost equal if the symmetric difference $R+ S$ is $H$-finite.
  We write $R =^a S$.  A set $A$ is $H$-almost invariant if $A =^a Ag$ for every $g \in G$.   If $A$ is  almost invariant, then so is 
  $A^* = G - A$.  We say that $A$ is proper if neither $A$ not $A^*$ is $H$-finite.
  
  \begin {conj} [Kropholler, 1988] \label {KC}  Let $G$ be a group and let $H$ be a subgroup. If there is a proper $H$-almost invariant set $A$ such that $A= AH$, then $G$ has a non-trivial action on a tree in which  $H$ fixes a vertex $v$ and every edge incident with $v$ has an $H$-finite stabiliser.
  \end {conj}
  
  If a finitely generated group $G$ has such an action, then it has such an action in which there is one orbit of edges.   This will correspond to a decomposition of $G$ as a free product with amalgamation or as an $HNN$-group.  We say that $G$ splits over an $H$-finite subgroup.
  
  In \cite{D}  it is shown that if $G$ is finitely generated over $H$, i.e. $G$ is generated   by $H\cup S$, where $S$ is finite,  then there is a connected $G$-graph $X$ with one orbit of vertices and finitely many orbits of edges, and for which there is a vertex $o$ fixed by $H$.   If there is a set $A$ as above, then the graph $H\backslash X$ has more than one end.    This means
 that $B(H\backslash X)$ has a nested set of generators, at least one of which will be an infinite subset of $V(H\backslash X)$, with infinite complement.  It is shown that lifting this set to $VX$ will
 give a proper $H'$-almost invariant subset $A'$ of $G$, with right stabiliser $K$, satisfying $H'A'K =A'$, with $H' \leq H\leq K$.
 
 This means that in Conjecture 2.1, we can replace the condition $A=AH$ by $A=HAK$ where $H \leq K$, in the case when $G$ is finitely generated over $H$.
   
  Let $A$ be $H$-almost invariant and let $HA =A$, so that $A+Ag$ is a union of finitely many right  $H$-cosets for every $g \in G$.
  Let $K$ be the right stabiliser of $A$.  Suppose that $H$ is the left stabiliser of $A$ and that $H\leq K$, so that $A = HAH = HAK$.

  Let $X$ be the $2$-complex in which  the $1$-skeleton is the complete graph on $VX = AG$,  and there is a $2$-simplex for each triple of distinct vertices in $VX$.
  Thus $X$ is the complete $2$-complex on the set $AG = \{ Ag | g \in G\}$.
  
  Let $\sigma = \sigma(u,v,w)$ be the $2$-simplex with vertices $u, v, w$.  The set $u+v$ is a finite set of cosets.  Each coset in $u+v$ also belongs to exactly one of $u+w$ and $v+w$,
  so that $u+v = (u+w)+(v+w)$.

  As in \cite {D} we can define a metric on $VX$ in which $d(u,v)$ is the number of right $H$-cosets in $u+v$.   
  
Consider a $2$-simplex $\sigma $ of $X$ with vertices $u,v, w$.  We know that $d(u,v) \leq d(u,w) + d(w, v)$.  Also $(u+v)+(v+w)+(w+u)   = 0 (mod \ 2)$ 
 and so the sum $d(u,v) + d(v+w) +d(w, u)$ is even.   Hence  as in \cite {DD} p224 there is a {\it pattern} P in $|X|$ for which 
\begin  {itemize}
\item [(i)]  For each $2$-simplex $\sigma $  of $X$  $P\cap |\sigma | $ is a union of finitely many disjoint straight lines joining distinct faces of $\sigma$.
\item [(ii)] If $\gamma = [u,v]$ is a $1$-simplex of $X$, then $|\gamma |\cap P$ consists of $d(u,v)$ points.
\end {itemize}
  A {\it track} is a connected pattern.
  
 Let $EX$ be the set of $1$- simplexes of $X$.  Let $j_P : EX \rightarrow \cz$,  where  $j_P (\gamma )  $ is the number of points in $|\gamma |\cap P $.
 Two patterns $P, Q$ are {\it equivalent}  if  $j_P =j_Q$.    Two disjoint tracks are equivalent if and only if they bound a band,  a closed subset of $X$ that contains no vertex
 or central region.   
 
 Since  $H^1(X, \Z _2) = 0$ (see below) our complex $X$ admits a two colouring in which vertices $u, v$ have the same colour if $d(u,v)$ is even.  This colouring extends to $|X| -P$ (see  Fig 1)  
 in such a way that adjacent regions have different colours.
 
 As in \cite {DD} p225, let $D_P$ be the graph in which 
 
 $VD_P$ is the set of components of $|X|-P$,   
 
 $ED_P$ is the set of tracks of $P$

   
  \begin{figure}[htbp]
\centering  
\begin{tikzpicture}[scale=.8]

  \draw  (0,.0) -- (4, 0) -- (2, 4) -- (0,0)  ;

\draw [red]  (4,0) node {$\bullet $} ;
\draw [red]  (2,4) node {$\bullet $} ;

\draw [red] (1.5, 3) -- (2.25, 3.5) ;
\draw [red]  (1.25, 2.5) -- (2.75, 2.5) ;
\draw [red]  (.5, 0) -- (.4, .8) ;
\draw [red]  (.7, 0) --(.8, 1.6);
\draw [red]  (1.1, 0) -- (1, 2) ;
\draw [red] (3, 0) -- (3.5, 1) ;
\draw [red]  (2.5, 0) -- (3, 2) ;

 \draw  (6,.0) -- (10, 0) -- (8, 4) -- (6,0)  ;
\draw  [left] (6,0) node {$u$} ;
\draw  [left] (8,4) node {$v$} ;

\draw  [right ] (10,0) node {$w$} ;
\draw  [left] (0,0) node {$u$} ;
\draw  [left] (2,4) node {$v$} ;

\draw  [right ] (4,0) node {$w$} ;

\draw [red] (7.5, 3) -- (8.25, 3.5) ;
\draw [red]  (7.25, 2.5) -- (8.75, 2.5) ;
\draw [red]  (6.5, 0) -- (6.4, .8) ;
\draw [red]  (6.7, 0) --(6.8, 1.6);
\draw [red]  (7.1, 0) -- (7, 2) ;
\draw [red] (9, 0) -- (9.5, 1) ;
\draw [red]  (8.5, 0) -- (9, 2) ;

\draw [red]  (10,0) node {$\bullet $} ;
\draw [red]  (8,4) node {$\bullet $} ;

\fill[pink] (9.5,1) --(10,0) -- (9, 0)--cycle ;
\fill[pink] (8.25,3.5) --(8,4) -- (7.5,3)--cycle ;
\fill[pink] (6.5,0) --(6.4,.8) -- (6.8,1.6)--(6.7,0) --cycle ;

\fill[pink] (7.25, 2.5) -- (8.75, 2.5) -- (9,2) --(8.5,0) -- (8.5, 0)--(7.1,0)--(7,2) --cycle ;
\end{tikzpicture}
\caption {\label 1}

\end{figure}
If $e \in ED_P$, then the vertices of $e$ are the components whose closures contain $e$.   Because  $X$ can be two coloured as above, every edge of $D_P$ is incident with different vertices,  so that there are no twisted tracks.  In fact $D_P$ is a tree, since $H^1(X, \Z _2) = 0$, which means that tracks separate by \cite {DD} Proposition VI, 4.3.
 We get  $H^1(X, \Z _2) = 0$, since  every $n$-cycle in the $1$-skeleton of $X$ bounds a disc made up of $n-2$ triangles, each containing a a particular vertex $v$ of the cycle and one of the $n-2$ edges of the cycle not incident with $v$.

   The pattern $P$ is invariant under the action of $G$, and so $D_P$ is a $G$-tree.
  We have shown that for an  $H$-almost invariant set $A$ such that $A =HAK$ as above, there is a uniquely determined $G$-tree given by the pattern $P$.

  Let $Hx$ be a coset.   If $Hx$ is in an edge $uv$ of the $2$-simplex $\sigma = uvw$,  then it belongs to exactly one of the other two sides.  For each edge containing $Hx$ select an interior point. Joining points corresponding to $Hx$
  in every $2$-simplex will give a track $t$ intersecting each simplex at most once.    If every track of $P$ was such a track, then $D_P$ would be a tree as predicted in
  Conjecture \ref {KC}.

   A `proof' was given of this in the first version of this paper.  
   However Sam Shepherd has pointed out that a track of $P$ might intersect an edge of $X$ in more than one point.
   To see how this can happen consider four distinct vertices $u, v, w, z$ in $VX$. They determine a subcomplex. $Y$  of $X$ consisting of $4$ $2$-simplexes that are the faces of a $3$-simplex (tetrahedron) with those four vertices.  Topologically this will be a $2$-sphere, and a track will intersect the $2$-sphere in finitely many simple closed curves.
 
   A track of $P$ might intersect $Y$ in a pattern in which there is a track as in Fig 2.

     In the first version of this paper it was claimed that this tree  had  the properties required to prove Conjecture 2.1.  This was done by showing that tracks corresponded to cosets in a nice way.

   In fact this will be the case if  every track intersects any edge in at most one point,  but so far we have been unable to show that this is the case.

 The fact that each track separates, means that there is another metric $d_1$ on $VX$ in which $d_1(u,v)$ is the number of tracks of $P$ which separate $u,v$.
 It can be seen that for every pair $u,v$ we have $d_1(u,v) \leq d(u,v)$.    The metric $d_1$ is a tree metric, since $d_1(u,v) $ is the number of edges in the tree $D_P$ joining $u,v$.
 
The metric $d_1$ will give rise to tracks that intersect each edge as most once.   For example consider the tracks for $d$ and $d_1$ intersecting $Y$.
 A  track in $Y$ is a simple closed curve, and will partition $VY$ into two subsets and each partition will be the same as one corresponding to a track which intersects each edge at most once. Thus for example the partition $\{ u,v,w,z\} = \{ u, v\} \cup \{ w,z\} $ given by the track in Fig 2 can be achieved by the track as in Fig 3,
and  the partition $\{ u,v,w,z\} = \{ u, z\} \cup \{ v,w\} $ can be achieved by the track as in Fig 4.

There is a bijection between the set of tracks for the metric $d_1$ and the set of tracks for the metric $d$.   To see this note that a track of $P$ separates $u,v$ if and only if it intersects
the edge $uv$ an odd number of times.  Since there are no twisted tracks,  every track intersects at least one edge an odd number of times.



Suppose that the $G$-tree $T$ corresponding to the $d_1$-metric for $P$ is trivial.   Thus there will be a vertex $o$ of  $T$ which is fixed by $G$.  If $o$ is in the orbit of the vertices of $X$, then $T$ will consist of a single point.  If not then $T$ will  have an orbit $uG$ of vertices of $X$ and there will be a vertex  $o$ of $T$ fixed by $G$ so that the $d_1$ distance of $o$  from every vertex of $uG$ is constant.  It seems likely that this means that  the only possibility for the intersection of $Y$ with the pattern corresponding to $d_1$ is a multiple of the  four tracks as in Fig 5. 
The action on $T$ is trivial, if and only if the values taken by $d_1$ are bounded.    If the action is trivial, then $d_1$ can take only even values.

  \begin{figure}[htbp]
\centering
\begin{tikzpicture}[scale=.8]

  \draw  (0,0) -- (0,4) ;
  \draw  (4,0) -- (4,4)--(0,4) ;
   \draw  (0,0) -- (4,4) ;
  \draw  (4,0)--(0,0) ;
  
\draw   (0,0) node {$\bullet $} ;
\draw   (4,0) node {$\bullet $} ;	
\draw  (0,4) node {$\bullet $} ;
\draw  (4,4) node {$\bullet $} ;
\
 \draw  (6,0)--(6,4) --(10,0)--(10,4) ; 

\draw   (6,0) node {$\bullet $} ;
\draw   (10,0) node {$\bullet $} ;
\draw  (6,4) node {$\bullet $} ;
\draw  (10,4) node {$\bullet $} ;

\draw [red] (6, 2.4) --(10, 2) ;
\draw [red]  (6,2.2) --(10,1.8) ;
\draw [red] (6, 2) --(10,1.6) ;

\draw [red] (6,1.8)--(8,0) ;

\draw [red] (8,4)--(10,2.2) ;


\draw [red] (0,2.4) -- (2, 4) ; 
\draw  [red] (2, 0) -- (4, 1.6) ; 

\draw  [red] (0,1.8)--(4,1.8) ;

\draw [red] (0,2)--(4,2) ;
\draw  [red] (0,2.2)--(4,2.2) ;

\draw (6,0) --(10,0) ;

\draw (6,4) --(10,4) ;

\draw  [left] (0,0) node {$u$} ;
\draw  [left] (6,0) node {$u$} ;
\draw  [left] (0,4) node {$v$} ;
\draw  [left] (6,4) node {$v$} ;

\draw  [right ] (4,0) node {$w$} ;
\draw  [right ] (10,0) node {$w$} ;
\draw  [right ] (4,4) node {$z$} ;

\draw  [right ] (10,4) node {$z$} ;

\end{tikzpicture}

\caption {\label 2}

\end{figure}

  \begin{figure}[htbp]
\centering
\begin{tikzpicture}[scale=.8]

  \draw  (0,0) -- (0,4) ;
  \draw  (4,0) -- (4,4)--(0,4) ;
   \draw  (0,0) -- (4,4) ;
  \draw  (4,0)--(0,0) ;
  
\draw   (0,0) node {$\bullet $} ;
\draw   (4,0) node {$\bullet $} ;	
\draw  (0,4) node {$\bullet $} ;
\draw  (4,4) node {$\bullet $} ;

 \draw  (6,0)--(6,4) --(10,0)--(10,4) ; 

\draw   (6,0) node {$\bullet $} ;
\draw   (10,0) node {$\bullet $} ;
\draw  (6,4) node {$\bullet $} ;
\draw  (10,4) node {$\bullet $} ;

\draw [red] (8, 4) --(8,0) ;
\draw [red] (2, 4) --(2,0) ;


\draw (6,0) --(10,0) ;

\draw (6,4) --(10,4) ;

\draw  [left] (0,0) node {$u$} ;
\draw  [left] (6,0) node {$u$} ;
\draw  [left] (0,4) node {$v$} ;
\draw  [left] (6,4) node {$v$} ;

\draw  [right ] (4,0) node {$w$} ;
\draw  [right ] (10,0) node {$w$} ;
\draw  [right ] (4,4) node {$z$} ;

\draw  [right ] (10,4) node {$z$} ;

\end{tikzpicture}

\caption {\label 2}

\end{figure}

 \begin{figure}[htbp]
\centering
\begin{tikzpicture}[scale=.8]

  \draw  (0,0) -- (0,4) ;
  \draw  (4,0) -- (4,4)--(0,4) ;
   \draw  (0,0) -- (4,4) ;
  \draw  (4,0)--(0,0) ;
  
\draw   (0,0) node {$\bullet $} ;
\draw   (4,0) node {$\bullet $} ;	
\draw  (0,4) node {$\bullet $} ;
\draw  (4,4) node {$\bullet $} ;

 \draw  (6,0)--(6,4) --(10,0)--(10,4) ; 

\draw   (6,0) node {$\bullet $} ;
\draw   (10,0) node {$\bullet $} ;
\draw  (6,4) node {$\bullet $} ;
\draw  (10,4) node {$\bullet $} ;

\draw [red] (8, 4) --(10,2) ;
\draw [red] (2, 0) --(4,2) ;
\draw [red] (8, 0) --(6,2) ;
\draw [red] (2, 4) --(0,2) ;


\draw (6,0) --(10,0) ;

\draw (6,4) --(10,4) ;

\draw  [left] (0,0) node {$u$} ;
\draw  [left] (6,0) node {$u$} ;
\draw  [left] (0,4) node {$v$} ;
\draw  [left] (6,4) node {$v$} ;

\draw  [right ] (4,0) node {$w$} ;
\draw  [right ] (10,0) node {$w$} ;
\draw  [right ] (4,4) node {$z$} ;

\draw  [right ] (10,4) node {$z$} ;

\end{tikzpicture}

\caption {\label 2}

\end{figure}

 \begin{figure}[htbp]
\centering
\begin{tikzpicture}[scale=.8]

  \draw  (0,0) -- (0,4) ;
  \draw  (4,0) -- (4,4)--(0,4) ;
   \draw  (0,0) -- (4,4) ;
  \draw  (4,0)--(0,0) ;
  
\draw   (0,0) node {$\bullet $} ;
\draw   (4,0) node {$\bullet $} ;	
\draw  (0,4) node {$\bullet $} ;
\draw  (4,4) node {$\bullet $} ;

 \draw  (6,0)--(6,4) --(10,0)--(10,4) ; 

\draw   (6,0) node {$\bullet $} ;
\draw   (10,0) node {$\bullet $} ;
\draw  (6,4) node {$\bullet $} ;
\draw  (10,4) node {$\bullet $} ;

\draw [red] (8.5, 4) --(10,2.5) ;
\draw [red] (2.5, 0) --(4,1.5) ;
\draw [red] (7.5, 0) --(6,1.5) ;
\draw [red] (1.5, 4) --(0,2.5) ;

\draw [red] (8.5, 0) --(10,1.5) ;
\draw [red] (2.5, 4) --(4,2.5) ;
\draw [red] (7.5, 4) --(6,2.5) ;
\draw [red] (1.5, 0) --(0,1.5) ;


\draw (6,0) --(10,0) ;

\draw (6,4) --(10,4) ;

\draw  [left] (0,0) node {$u$} ;
\draw  [left] (6,0) node {$u$} ;
\draw  [left] (0,4) node {$v$} ;
\draw  [left] (6,4) node {$v$} ;

\draw  [right ] (4,0) node {$w$} ;
\draw  [right ] (10,0) node {$w$} ;
\draw  [right ] (4,4) node {$z$} ;

\draw  [right ] (10,4) node {$z$} ;

\end{tikzpicture}

\caption {\label 2}

\end{figure}

 The tracks for the metric $d_1$  intersect  each $2$-simplex in just one  line joining two edges.

  \begin{figure}[htbp]
\centering
\begin{tikzpicture}[scale=.8]

  \draw  (0,0) -- (0,4) --(4,4)-- (0,0)  ;
  \draw  (0,0) --(4,0) -- (4,4) ;

\draw [red ] (4, 1.5) --(0,1.5) ;
\draw [red ] (4, 1.7) --(0,1.7) ;
\draw [red ] (4, 1.9) --(0,1.9) ;
\draw [red ] (4, 2.1) --(0,2.1) ;
\draw [red ] (4, 2.3) --(0,2.3) ;
\draw [red ] (4, 2.5) --(0,2.5) ;

\draw   (0,0) node {$\bullet $} ;
\draw   (4,0) node {$\bullet $} ;
\draw  (0,4) node {$\bullet $} ;
\draw  (4,4) node {$\bullet $} ;
\
 \draw  (6,0)--(6,4) --(10,0)--(10,4) ; 

\draw   (6,0) node {$\bullet $} ;
\draw   (10,0) node {$\bullet $} ;
\draw  (6,4) node {$\bullet $} ;
\draw  (10,4) node {$\bullet $} ;

\draw [red ] (6, 1.5) --(10,1.5) ;
\draw [red ] (6, 1.7) --(10,1.7) ;
\draw [red ] (6, 1.9) --(10,1.9) ;
\draw [red ] (6, 2.1) --(10,2.1) ;
\draw [red ] (6, 2.3) --(10,2.3) ;
\draw [red ] (6, 2.5) --(10,2.5) ;

\draw [red] (6,2.7)--(7.3,4) ;
\draw [red] (6,3.1)--(6.9,4) ;
\draw [red] (6,3.5)--(6.5,4) ;

\draw [red] (0,2.7)--(1.3,4) ;
\draw [red] (0,3.1)--(0.9,4) ;
\draw [red] (0,3.5)--(0.5,4) ;

\draw [red] (4, 3.8) --(3.8, 4) ;
\draw [red] (4, 3.4) --(3.4, 4) ;
\draw [red] (4, 3) --(3, 4) ;
\draw [red] (4, 2.6) --(2.6,4) ;

\draw [red ] (10, 3.8) --(9.8,4) ;
\draw [red ] (10, 3.4) --(9.4,4) ;
\draw [red ] (10, 3) --(9,4) ;
\draw [red ] (10, 2.6) --(8.6,4) ;

\draw [red] (10,1.3)--(8.7,0) ;
\draw [red] (10,0.9)--(9.1,0) ;
\draw [red] (10,0.5)--(9.5,0) ;

\draw [red] (6,.2)--(6.2,0) ;
\draw [red] (6,1)--(7,0) ;

\draw [red] (6,.6)--(6.6,0) ;

\draw [red] (6,1.4)--(7.4,0) ;
\draw [red] (0,1.4)--(1.4,0) ;
\draw [red] (0,.2)--(0.2,0) ;
\draw [red] (0,1)--(1,0) ;

\draw [red] (0,.6)--(.6,0) ;

\draw [red] (4, 1.3) --(2.7,0) ;
\draw [red] (4, .9) --(3.1,0) ;
\draw [red] (4, .5) --(3.5,0) ;

\

\draw [red] (6,.2)--(6.2,0) ;
\draw [red] (6,1)--(7,0) ;

\draw [red] (6,.6)--(6.6,0) ;

\draw [red] (6,1.4)--(7.4,0) ;
\draw [red] (6,.2)--(6.2,0) ;
\draw [red] (6,1)--(7,0) ;

\draw [red] (6,.6)--(6.6,0) ;

\draw [red] (6,1.4)--(7.4,0) ;

\draw (6,0) --(10,0) ;

\draw (6,4) --(10,4) ;

\draw  [left] (0,0) node {$u$} ;
\draw  [left] (6,0) node {$u$} ;
\draw  [left] (0,4) node {$v$} ;
\draw  [left] (6,4) node {$v$} ;

\draw  [right ] (4,0) node {$w$} ;
\draw  [right ] (10,0) node {$w$} ;
\draw  [right ] (4,4) node {$z$} ;

\draw  [right ] (10,4) node {$z$} ;

\end{tikzpicture}
\caption {\label 2}
\end{figure}

If every intersection of  the pattern $P$ with $Y$ is as in Figures 3, 4, 5 or 6, then one can prove that the metrics $d$ and $d_1$ are the same.  This is because if there are no intersections of tracks with $Y$ as in Fig 2, then it is the case that  every track intersects each edge at most once and  there is a bijection between
tracks and cosets. 



If the tree $D_P$ is non-trivial,  it will be a $G$-tree.  If, in addition,  the metrics $d$ and $d_1$ are the same, then this tree will be as predicted in Conjecture \ref {KC}.  The vertex of $D_P$ corresponding to the vertex $A$ of $X$ will have stabiliser $K$
 which contains $H$.
 
 For $H$ finite, the Almost Stability Theorem \cite {DD} tells us that the almost equality class of $A$ is the vertex set of a $G$-tree with finite edge groups.  If $A$ is a proper almost invariant set,  then this set does not include either the empty set or all of $G$, and so this tree is non-trivial.  This will also be the case for our tree $D_P$.


\begin{thebibliography}{99}

\bibitem {DD} Warren Dicks and M.J.Dunwoody.  \emph {Groups acting on graphs}. Cambridge Studies in Advanced Mathematics {\bf 17}  (1989).

\bibitem {D} M.J.Dunwoody. \emph {Structure trees, networks and almost invariant sets.}  \emph { Groups graphs and random walks} LMS lecture Notes {\bf 436} (2017) 137-175.





\end{thebibliography}
 \end{document}